\documentclass[12pt,leqno]{amsart}
\usepackage{amsmath,amsfonts,amssymb,amscd,amsthm,amsbsy,upref}
\title[Ball covering property from commutative to non-commutative spaces]{Ball covering property from commutative function spaces to non-commutative spaces of operators}

\author{Minzeng Liu}
\address{School of Mathematical Sciences and LPMC, Nankai University, Tianjin
300071, P.R. China}
\email{2120200027@mail.nankai.edu.cn}

\author{Rui Liu}
\address{School of Mathematical Sciences and LPMC, Nankai University, Tianjin
300071, P.R. China}
\email{ruiliu@nankai.edu.cn}

\author{Jimeng Lu}
\address{School of Mathematics and Statistics, University of New South Wales,
Sydney, NSW 2052, Australia}
\address{School of Mathematical Sciences and LPMC, Nankai University, Tianjin
300071, P.R. China}
\email{jimeng.lu@unsw.edu.au, jimenglu@mail.nankai.edu.cn}

\author{Bentuo Zheng}
\address{Department of Mathematics, The University of Memphis, Memphis, TN 38152, USA}
\email{bzheng@memphis.edu}

\thanks{Rui Liu was partially supported by the
National Natural Science Foundation of China (No. 11671214, 11971348, 12071230)
and hundred young academia leaders program of Nankai University.
Bentuo Zheng was supported in part by Simons Foundation Grant 585081}

\date{}
\newtheorem{thm}{Theorem}[section]
\newtheorem{cor}{Corollary}[section]

\newtheorem{lemma}{Lemma}[section]
\newtheorem{remark}{Remark}[section]

\newtheorem{defn}[thm]{Definition}

\def\supp{\operatorname{supp}}

\newcommand{\N}{\mathbb{N}}

\newcommand{\R}{\mathbb{R}}

\newcommand{\cB}{\mathcal{B}}

\newcommand{\cF}{\mathcal{F}}

\newcommand{\cK}{\mathcal{K}}

\begin{document}
\maketitle
\begin{abstract}
A Banach space is said to have the ball-covering property (abbreviated BCP) if its unit sphere can be covered by countably many closed, or equivalently, open balls off the origin. Let $K$ be a locally compact Hausdorff space and $X$ be a Banach space. In this paper, we give a topological characterization of BCP, that is, the continuous function space $C_0(K)$ has the (uniform) BCP if and only if $K$ has a countable $\pi$-basis. Moreover, we give the stability theorem: the vector-valued continuous function space $C_0(K,X)$ has the (strong or uniform) BCP if and only if $K$ has a countable $\pi$-basis and $X$ has the (strong or uniform) BCP. We also explore more examples for BCP on non-commutative spaces of operators $\cB(X,Y)$. In particular, these results imply that $\cB(c_0)$, $\cB(\ell_1)$ and every subspaces containing finite rank operators in $\mathcal{B}(\ell_p)$ for $1< p<\infty$ all have the BCP, and $\mathcal B(L_1[0,1])$ fails the BCP. Using those characterizations and results, we show that BCP is not hereditary for 1-complemented subspaces (even for completely 1-complemented subspaces in operator space sense) by constructing two different counterexamples.

\end{abstract}

\section{Introduction}
How to cover subsets of normed spaces by different shapes of open sets is one of the most intensively investigated problems in the study of geometric and topological properties of these spaces.  A normed space $X$ is said to have the \textit{ball-covering property}  (\textit{BCP}) \cite{C} if its unit sphere can be covered by countably many closed, or equivalently, open balls, none of which contains the origin of $X$. The centers of the balls are called the \textit{BCP points} of $X$. It is worth noting that there are two frequently used variants of this definition. Suppose that $\{B_n\}$ is a sequence of balls off the origin whose union cover the unit sphere of $X$. Then $X$ is said to have the \textit{strong ball-covering property} (\textit{SBCP}) \cite{LZh2}, if the radii of $\{B_n\}$ is a bounded sequence. $X$ is said to have the \textit{uniform ball-covering property} (\textit{UBCP}) \cite{LZh2}, if $X$ has the SBCP, and there exists a ball $B_0$ containing the origin such that $B_n\cap B_0=\emptyset$ for all $n\in\mathbb N$.

It is clear from the definition that all separable normed spaces have the BCP. While $\ell_\infty$ admits the BCP, it was shown by L. Cheng, Q. Cheng and X. Liu \cite{CCL} that it can be renormed to fail this property. As a consequence, the BCP is not preserved under linear isomorphisms. Furthermore, V. P. Fonf and C. Zanco \cite{FZ} proved that $X^*$ is $w^*$-separable if and only if $X$ can be ($1+\varepsilon$)-equivalently renormed to have the SBCP for all $\varepsilon>0$, which further reveals the fact that the BCP is a geometric property deeply related to the weak star topology for dual spaces. The BCP is also widely linked to a number of important properties, such as the $G_\delta$ property of points in $X$, Radon-Nikodym property \cite{CWWZ}, uniform convexity, uniform non-squareness, strict convexity and dentability \cite{SC2, SC3}, and universal finite representability and B-convexity \cite{Z}.

The stability of the BCP under operations of Banach spaces, such as product and direct sum, has been a focus in recent study. It was proved in \cite{LLW} that for $1\leq p\leq\infty$, $(X\times Y,||\cdot||_p)$ has the BCP if and only if $X$ and $Y$ have the BCP. Z. Luo and B. Zheng \cite{LZh1} proved that $\left(\sum\oplus X_k\right)_E$ has the BCP if and only if each summand $X_k$ has the BCP, where $E$ is either a Lorentz sequence space, a separable Orlicz sequence space, or $\ell_\infty$. It was also shown in [6] that if $(\Omega,\Sigma,\mu)$ is a separable measure space, then the space of Bochner integrable functions $L^p(\mu,X)$ has the BCP if and only if $X$ has the BCP.

Although a lot of progress has been made for the study of the BCP recently, it is still unknown for what locally compact Hausdorff space $K$, $C_0(K)$ has the BCP. Since the Banach-Stone theorem tells us that the topological properties of $K$ can be transferred to geometric properties of $C_0(K)$ spaces, the following more general question is natural.\\

\noindent
{\bf Question 1:} {\it Let $K$ be a locally compact Hausdorff space and $X$ be a Banach space. Can we find a topological property $P$ of $K$ so that $K$ has property $P$ and $X$ has the BCP if and only if $C_0(K, X)$ has the BCP}?\\

Considering the definition the BCP, we had an intuitive conjecture that $K$ has a countable topological basis if and only if $C_0(K)$ has the BCP. It is not very difficult to show that if $K$ has a countable topological basis, then $C_0(K)$ has the BCP. But we could not prove the converse until we realized that a countable $\pi$-basis is the right topological property.

\begin{defn}\cite{Tam, A.K.}
Let $K$ be a topological space. A collection of nonempty open subsets $\mathcal{U}=\{U_{\alpha}\}_{\alpha\in\mathcal{I}}$ of $K$ is called a $\pi$-basis (weak basis) for $K$ if every nonempty open subset of $K$ contains some $U_{\alpha}$ in $\mathcal{U}$.
\end{defn}

The notion of weak bases was introduced by Arhangel'skii \cite{A} to study symmetrizable spaces. The famous Nagata-Smirnov metrization theorem states that a topological space $K$ is metrizable if and only if it is regular, Hausdorff and has a countably locally finite basis. There are other important metrization theorems for topological spaces with weak bases given by R. E. Hodel \cite{H} and H. W. Martin \cite{M}.

The first goal of this paper is to give a complete answer to Question 1, which can be viewed as a geometric characterization of countable $\pi$-basis (weak bases) for topological spaces.

As mentioned earlier, the BCP is not preserved under isomorphisms. Moreover, it was proved in \cite{CCL} that the BCP does not pass to subspaces and quotients. If $X$ is the $\ell_p$ direct sum of $Y$ and $Z$, then $X$ has the BCP if and only $Y$ and $Z$ have the BCP. However whether the BCP passes to $1$-complemented subspaces is still open.\\

\noindent
{\bf Question 2:} {\it Let $X$ be a Banach space with the BCP and $Y$ be an $1$-complemented subspace of $X$. Does $Y$ have the BCP}?\\

Recall that $Y$ is said to be a complemented subspace of $X$ if there is a continuous linear projection from $X$ onto $Y$. If in addition, the projection can be chosen to have norm one, then $Y$ is an $1$-complemented subspace of $X$. The second goal of this paper is to provide a counterexample for Question 2. To be more precise, we show that there exists a Banach space $X$ with the UBCP so that there is an $1$-complemented $Y$ of $X$ which fails the BCP. The construction of the counterexample is based on the above topological characterization of the BCP.

In the last section, we explore more examples for the BCP on non-commutative spaces of operators on separable Banach spaces. We prove that all subspaces containing finite rank operators in $\cB(c_0)$ and $\cB(\ell_p)$ $(1\le p<\infty)$ have the BCP, and that $X^*$ has the BCP (so $X$ is separable) is a necessary condition for $\cB(X)$ to have the BCP. An immediate consequence is that $\cB(L^1[0,1])$ fails the BCP. Using either the topological characterization or those new examples on non-commutative spaces, we can show that the BCP does not pass to 1-complemented subspaces (or even completely 1-complemented subspaces in the operator space sense).

 In this paper, the letters $X$ and $Y$ will denote Banach spaces and $K$ will denote a locally compact Hausdorff space, and by $X\subset Y$ we mean that $X$ can be isometrically embedded into $Y$. For $x\in X$ and $r>0$, we denote the open ball centered at $x$ with radius $r$ by $B^\circ(x,r)$ and its closure by $B(x,r)$. The closed unit ball and unit sphere of $X$ will be denoted by $B_X$ and $S_X$ respectively.  $C_0(K,X)$ is the Banach space of continuous functions from $K$ into $X$ vanishing at infinity endowed with the supremum norm, and $\mathcal B(X, Y)$, $\cK(X,Y)$ and $\cF(X,Y)$ are the space of all bounded linear operators, compact operators and finite rank operators from $X$ into $Y$, respectively.

\section{Topological Characterization and $C_0(K,X)$ }

In this section, we shall make use of the following result in our proof.

\begin{lemma}[\cite{LZh1}, Lemma 2.2]\label{lem1.1}
	\it Let $X$ be a normed space and $x\in X\backslash\{0\}$. If $0<s<t$, then
$B^\circ(sx,||sx||)\subset B^\circ(tx,||tx||)$ and $||sx||-||y-sx||\leq||\,tx||-||\,y-tx||$ for $y\in X$.
\end{lemma}

\begin{thm}\label{thm1}
Let $K$ be a locally compact Hausdorff space. Then
the followings are equivalent:
\begin{enumerate}
  \item[(i)] $C_0(K)$ has the BCP;
  \item[(ii)] $K$ has a countable $\pi$-basis;
  \item[(iii)] $C_0(K)$ has the UBCP.
\end{enumerate}
\end{thm}

\begin{proof}
(ii)$\Rightarrow$(iii):
Suppose $K$ has a countable $\pi$-basis $(U_n)$. Let $\lambda>1$. Let $(f_n)$ be a sequence of functions in $C_0(K)$ so that $\|f_n\|=\lambda, f_n\geq 0$ and $\supp(f_n)\subset U_n$. Then the set of functions $(\pm f_n)$ are BCP points for $C_0(K)$. Actually for any $g\in C_0(K)$ with $\|g\|=1$, either $g^{-1}[-1, -1/2)$ or $g^{-1}(1/2, 1]$ is a nonempty open set. Without loss of generality, we assume $g^{-1}(1/2, 1]$ is nonempty. Then there exists some $U_n$ contained in $g^{-1}(1/2, 1]$. So we have
\begin{equation*}
\begin{split}
\|g-f_n\|&=\max\{\sup_{\tau\in U_n}|g(\tau)-f_n(\tau)|, \sup_{\tau\in U_n^c} |g(\tau)-f_n(\tau)|\}\\
&=\max\{\sup_{\tau\in U_n}|g(\tau)-f_n(\tau)|, \sup_{\tau\in U_n^c} |g(\tau)|\}\\
&\leq \max\{\sup_{\tau\in U_n}|g(\tau)-f_n(\tau)|, 1\}\\
&\leq \max\{\lambda-1/2, 1\}\\
&<\lambda.
\end{split}
\end{equation*}
Hence $C(K)$ has the UBCP.

(i)$\Rightarrow$(ii):
Now suppose that $K$ does not have a countable $\pi$-basis. Assume that $(g_n)$ is a set of BCP points for $C_0(K)$ with $\|g_n\|>1$ for each $n\in\mathbb{N}$. Let
$$
A_{n,k}=g_n^{-1}([-\|g_n\|, -\|g_n\|+1/k)\cup (\|g_n\|-1/k, \|g_n\|], \forall n, k\in\mathbb{N}.
$$
Since $(A_{n,k})$ cannot form a $\pi$-basis for $K$, there exists a nonempty open set $U\subset K$ so that $A_{n,k}\setminus U$ is nonempty for all $n, k\in\mathbb{N}$. Let $f$ be any function in $C_0(K)$ with $\supp(f)\subset U$. Then for all $n\in\mathbb{N}$, we have
\begin{equation*}
\begin{split}
\|f-g_n\|&=\max\{\sup_{\tau\in U}|f(\tau)-g_n(\tau)|, \sup_{\tau\in U^c} |f(\tau)-g_n(\tau)|\}\\
&\geq\sup_{\tau\in U^c} |f(\tau)-g_n(\tau)|\\
&=\sup_{\tau\in U^c} |g_n(\tau)|\\
&=\|g_n\|
\end{split}
\end{equation*}
This implies that $(g_n)$ is not a set of BCP points for $C_0(K)$ and we get a contradiction.
\end{proof}

Now we give the topological characterization of BCP for the vector-valued case.

\begin{thm}\label{thm2}
Let $K$ be a locally compact Hausdorff space and $X$ be a Banach space. Then
the followings are equivalent:
\begin{enumerate}
  \item[(i)] $C_0(K, X)$ has the BCP;
  \item[(ii)] $C_0(K)$ and $X$ both have the BCP;
  \item[(iii)] $K$ has a countable $\pi$-basis and $X$ has the BCP.
\end{enumerate}
\end{thm}

\begin{proof}
(i)$\Rightarrow$(ii): Suppose that $S_{C_0(K,X)}\subset\bigcup\limits_{n=1}^\infty B(F_n,r_n)$, where $r_n<\|F_n\|$ for all $n\in\mathbb N$. For each $n$, there exists $\tau_n\in K$ such that $\|F_n(\tau_n)\|=\|F_n\|$ since $F_n$ attains its norm on $K$.
We show that $\bigcup\limits_{n=1}^\infty B(F_n(\tau_n),r_n)$ is a countable ball covering of $S_X$ off the origin.	
	For $x\in S_X$, consider the function $F_x:K\rightarrow X$ given by $F_x(\tau)=x$. Obviously $F_x\in S_{C_0(K,X)}$ and there exists $n_0\in\mathbb N$ such that
$\|F_x-F_{n_0}\|\leq r_{n_0}$. Therefore
	\[\begin{aligned}
\|x-F_{n_0}(\tau_{n_0})\|&=\|(F_x-F_{n_0})(\tau_{n_0})\|\\
&\leq\|F_x-F_{n_0}\|\leq r_{n_0}<\|F_{n_0}(\tau_{n_0})\|.
\end{aligned}\]
Suppose that \[S_{C(K,X)}\subset\bigcup\limits_{n=1}^\infty B^\circ(F_n,\|F_n\|_{C_0(K,X)}).\] We will show that $\{\pm m\|F_n(\cdot)\|:m,n\in\mathbb N\}$ is a collection of BCP points of $C_0(K)$.
	For $f\in S_{C_0(K)}$, without loss of generality we may assume that there exists $\tau'\in K$ such that $f(\tau')=1$. Let $f^+=f\vee0$ and fix $x\in S_X$. Then $f^+(\cdot)x\in S_{C(K)}$ and there exists $n_0\in\mathbb N$ such that \[\|f^+(\cdot)x-F_{n_0}\|_{C_0(K,X)}<\|F_{n_0}\|_{C_0(K,X)}.\] Denote
	\[K^+=\{\tau\in K:f(\tau)\geq0\}, K^-=\{\tau\in K:f(\tau)\leq0\}.\]
	Since $f$ is continous, both $K^+$ and $K^-$ are closed and locally compact subsets of $K$. Thus
	\[\begin{aligned}
	\|f^+(\cdot)x-F_{n_0}\|_{C_0(K,X)}&=\max\left\{\max_{\tau\in K^+}\|f^+(\tau)x-F_{n_0}(\tau)\|_X,\max_{\tau\in K^-}\|f^+(\tau)x-F_{n_0}(\tau)\|_X\right\}\\
	&=\max\left\{\max_{\tau\in K^+}\|f^+(\tau)x-F_{n_0}(\tau)\|_X,\max_{\tau\in K^-}\|F_{n_0}(\tau)\|_X\right\}\\
	&<\|F_{n_0}\|_{C_0(K,X)}\\&=\max\left\{\max_{\tau\in K^+}\|F_{n_0}(\tau)\|_X,\max_{\tau\in K^-}\|F_{n_0}(\tau)\|_X\right\},
	\end{aligned}\]
	which implies that
	\[\max_{\tau\in K^+}\|F_{n_0}(\tau)\|_X>\max_{\tau\in K^-}\|F_{n_0}(\tau)\|_X.\]
	Therefore there exists a sufficiently large $m_0\in\mathbb N$ such that
	\[m_0\cdot\max_{\tau\in K^+}\|F_{n_0}(\tau)\|_X>m_0\cdot\max_{\tau\in K^-}\|F_{n_0}(\tau)\|_X+1.\]
	It suffices to show that
	\[\big\|f-m_0\|F_{n_0}(\cdot)\|\big\|_{C_0(K)}<\big\|m_0\|F_{n_0}(\cdot)\|\big\|_{C_0(K)}.\]
	For $\tau\in K^+$, it follows from Lemma \ref{lem1.1} that
	\[\begin{aligned}
	\big|f(\tau)-m_0\|F_{n_0}(\tau)\|\big|&=\big|\|f^+(\tau)x\|_X-m_0\|F_{n_0}(\tau)\|_X\big|\\
	&\leq\|f^+(\tau)x-m_0F_{n_0}(\tau)\|_X\\&\leq\|f^+(\cdot)x-m_0F_{n_0}\|_{C_0(K,X)}\\
	&<m_0\|F_{n_0}\|_{C_0(K,X)}\\&=\big\|m_0\|F_{n_0}(\cdot)\|\big\|_{C_0(K)}.
	\end{aligned}\]
	For $\tau\in K^-$ (note that $|f(\tau)|\leq1$),
	\[\begin{aligned}
	\big|f(\tau)-m_0\|F_{n_0}(\tau)\|\big|&=|f(\tau)|+m_0\|F_{n_0}(\tau)\|_X\\&\leq1+m_0\cdot\max_{\tau\in K^-}\|F_{n_0}(\tau)\|_X\\
	&<m_0\cdot\max_{\tau\in K^+}\|F_{n_0}(\tau)\|_X\\&=\big\|m_0\|F_{n_0}(\cdot)\|\big\|_{C_0(K)}.
	\end{aligned}\]

(ii)$\Leftrightarrow$(iii) by Theorem \ref{thm1}.

(iii)$\Rightarrow$(i):
Suppose that $K$ has a countable $\pi$-basis and $X$ has the BCP. Let $(U_n)$ be a countable $\pi$-basis for $K$ and $(x_n)$ be a sequence of BCP points for $X$ so that $\|x_n\|>1$ and $S_X\subset \cup B(x_n, r_n)$ with $1<r_n<\|x_n\|$ for each $n\in\mathbb{N}$. Let $(f_n)$ be a sequence of functions in $C(K)$ so that $\|f_n\|=1, f_n\geq 0$ and $\supp(f_n)\subset U_n$. We will show that $(f_n x_k)_{n, k}$ is a set of BCP points for $C_0(K, X)$. Let $g$ be any function in $C(K, X)$ with $\|g\|=1$. Since $g$ is continuous and $K$ is compact, there is an $x_g=g(\tau_0)$ for some $\tau_0\in K$ so that $\|x_g\|=1$. There is an $t\in\mathbb{N}$ such that $\|x_t-x_g\|\leq r_t$. Consider the open set $g^{-1}(B^\circ(x_g, \frac{\|x_t\|-r_t}{2}))$. Since $(U_n)$ is a $\pi$-basis of $K$, there exists $m\in\mathbb{N}$ so that \[U_m\subset g^{-1}\Big(B^\circ(x_g, \frac{\|x_t\|-r_t}{2})\Big).\] So we have
\begin{equation*}
\begin{split}
\|f_mx_t-g\|&=\max\{\sup_{\tau\in U_m}\|f_m(\tau)x_t-g(\tau)\|, \sup_{\tau\in U_m^c} \|f_m(\tau)x_t-g(\tau)\|\}\\
&=\max\{\sup_{\tau\in U_m}\|f_m(\tau)x_t-g(\tau)\|, \sup_{\tau\in U_m^c} \|g(\tau)\|\}\\
&\leq\max\{\sup_{\tau\in U_m}(\|f_m(\tau)x_t-x_g\|+\|x_g-g(\tau)\|), 1\}\\
&\leq\max\{\sup_{\tau\in U_m}(\|f_m(\tau)x_t-x_g\|+\frac{\|x_t\|-r_t}{2}), 1\}
\end{split}
\end{equation*}
If $f_m(\tau)=1$, then \[\|f_m(\tau)x_t-x_g\|=\|x_t-x_g\|\leq r_t.\] If $0\leq f_m(\tau)<1$, then
\begin{equation*}
\begin{split}
\|f_m(\tau)x_t-x_g\|&=\|f_m(\tau)(x_t-x_g)+(f_m(\tau)-1)x_g\|\\
&\leq\|f_m(\tau)(x_t-x_g)\|+\|(f_m(\tau)-1)x_g\|\\
&=f_m(\tau)\|x_t-x_g\|+(1-f_m(\tau))\\
&\leq f_m(\tau) r_t+(1-f_m(\tau))\\
&< f_m(\tau) r_t+(1-f_m(\tau))r_t\\
&=r_t
\end{split}
\end{equation*}
Therefore we get
$$
\|f_mx_t-g\|\leq \max\{(r_t+\frac{\|x_t\|-r_t}{2}), 1\}=\max\{\frac{\|x_t\|+r_t}{2}, 1\}<\|x_t\|.
$$
\end{proof}

By the above topological characterization of the BCP, it is easy to see that the (uniform) BCP preserves for continuous function spaces on topological product of compact Hausdorff spaces.

\begin{cor}
Let $\{\Omega_k\}$ be a finite or countable sequence of compact Hausdorff spaces. Let $\times \Omega_k$ be the product space. Then the followings are equivalent:
\begin{enumerate}
  \item[(i)] $\times \Omega_k$ has a countable $\pi$-basis;
  \item[(ii)] $C(\times \Omega_k)$ has the (uniform) BCP;
  \item[(iii)] $C(\Omega_k)$ has the (uniform) BCP for each $k$;
  \item[(iv)] $\Omega_k$ has a countable $\pi$-basis for each $k$.
\end{enumerate}
\end{cor}

Now we give a counterexample to Question 2 from our topological characterization of BCP on the spaces $C(K)$. Let L be the space of all real valued functions $f: \R \rightarrow \{0, 1\}$,
endowed with the pointwise topology, which can be viewed as a product space \[L=\{0,1\}^\R.\]
Then it is compact by  Tychonoff's theorem.

We claim that $L$ fails to have a countable $\pi$-basis. If not, there is
a countable $\pi$-basis $\{U_n\}$ of $L$. Let $W_r=\{f\in L : f(r)=1\}$
for $r\in\R$ be open sets in $L$. By the definition of $\pi$-basis, for every $r\in\R$ there is $n\in\N$
such that $U_n\subseteq W_r$. Then by a cardinality argument, there exist $n_0\in\N$
and an uncountable subset $A\subseteq\R$ such that
\[U_{n_0}\subseteq \bigcap_{r\in A} W_r.\]
This leads to a contradiction, since $\bigcap_{r\in A} W_r$ has empty interior for any infinite subset $A\subseteq \R$.

Please notice that $L$ is separable. Actually, it is easy to see that
the set of characteristic functions of finite unions of intervals with rational endpoints form a countable dense subset of $L$.

Now let us fix $\{h_n\}_{n\in\N}$ a countable dense subset of $L$, and let
\[K=L\times \{0\} \bigcup \Big\{\Big(h_n, \frac{1}{n}\Big):n\in\N\Big\}\subseteq L\times\R.\]
Clearly, it is a closed subset of $L\times [0,1]$, and therefore a compact space.
The points of the form $(h_n, \frac{1}{n})$ are isolated in $K$ and form a countable dense subset of $K$.
In particular, $K$ has a countable $\pi$-basis, made by the singletons $\{(h_n, \frac{1}{n})\}$.

\begin{remark}\label{remark3.1}
By Theorem \ref{thm1}, it follows that $C(K)$ has the (uniform) BCP, but $C(L)$
fails the BCP. If we prove that $C(L)$ is a $1$-complemented subspace of $C(K)$, then this gives a
counterexample to Question 2.
\end{remark}
\begin{proof}
Let $\alpha: K\rightarrow L$ be the projection onto the first coordinate,
that is, $\alpha(f,x)=f$ for all $(f,x)\in K.$ and $\beta: L\rightarrow K$
given by $\beta(f)=(f,0)$ for $f\in L.$ Both the above are continuous maps with
$\alpha \circ \beta =\mathrm{id}_L.$ Look at the corresponding composition operators
\[T_\alpha: C(L) \rightarrow C(K),\quad T_\alpha(h)=h\circ\alpha, \forall\, h\in C(L),\]
\[T_\beta: C(K) \rightarrow C(L),\quad T_\beta(g)=g\circ\beta, \forall\, g\in C(K).\]
It is easy to prove that these are operators of norm 1 and
\[T_\beta(T_\alpha(h))=T_\beta(h\circ \alpha)=h\circ \alpha \circ \beta=h, \forall\, h\in C(L).\]
That is $T_\beta \circ T_\alpha=\mathrm{id}_{C(L)}.$ Thus, $C(L)$
is a 1-complemented subspace of $C(K)$.
\end{proof}

\begin{remark}\label{remark3.2}
From the map $\beta$, we know that although $K$ is a compact Hausdorff space with a countable $\pi$-basis and $L$ is a topological quotient of $K$, L does not have a countable $\pi$-basis.
\end{remark}

An analogous stability result also holds for the strong BCP and the uniform BCP. Although the main idea of the proof is similar to that of Theorem \ref{thm2}, there are some technical differences. For the convenience of the readers, we include the proof here.

\begin{thm}\label{thm2.sbcp}
Let $K$ be a locally compact Hausdorff space and $X$ be a Banach space. Then
the following are equivalent:
\begin{enumerate}
  \item[(i)] $C_0(K, X)$ has the UBCP (SBCP);
  \item[(ii)] $C_0(K)$ and $X$ both have the UBCP (SBCP);
  \item[(iii)] $K$ has a countable $\pi$-basis and $X$ has the UBCP (SBCP).
\end{enumerate}
\end{thm}
\begin{proof}
We deal with the UBCP case ($r^*=0$ corresponds to the case of SBCP).

(i)$\Rightarrow$(ii): By Theorem \ref{thm1} and \ref{thm2}, we only need to prove that $X$ has the UBCP. Suppose that $S_{C_0(K,X)}\subset\cup_{n=1}^\infty B(F_n,r_n)$, and there are constants $M >2$ such that $r_n<M$ for all $n\in\mathbb N$ and $r^{*}\ge 0$ ($r^{*}=0$ corresponds to the case of SBCP) satisfying \[\bigcup_{n=1}^\infty B(F_n,r_n)\bigcap\ B(0,r^{*})=\emptyset.\]
We can use Lemma \ref{lem1.1} and the fact $\{r_n\}_{n=1}^\infty$ is bounded to
obtain a constant $r^{**}$  and a sequence$(G_n)_{n=1}^\infty$ that satisfy $r^{**} > 2+2M$, $\|G_n\|=r^{**}$ and \[\bigcup_{n=1}^\infty B(G_n,r^{**}-r^{*})\bigcap\ B(0,r^{*})=\emptyset.\]   For each $n$, there exists $\tau_n\in K$ such that $\|G_n(\tau_n)\|=r^{**}$ since $G_n$ attains its norm on $K$. We wil show that$\cup_{n=1}^\infty B(G_n(\tau_n),r^{**}-r^{*})$ is a uniform ball covering of $S_X$ .	

	For $x\in S_X$, consider the function $F_x:K\rightarrow X$ given by $F_x(\tau)=x$. Obviously $F_x\in S_{C_0(K,X)}$ and there exists $n_0\in\mathbb N$ such that
$\|F_x-G_{n_0}\|\leq r^{**}-r^{*}$. Therefore
\[
\|x-G_{n_0}(\tau_{n_0})\|=\|(F_x-G_{n_0})(\tau_{n_0})\|
\leq\|F_x-G_{n_0}\|\leq r^{**}-r^{*}.\]

(ii)$\Leftrightarrow$(iii) by Theorem \ref{thm1} again.

(iii)$\Rightarrow$(i):
Suppose that $K$ has a countable $\pi$-basis and $X$ has the UBCP. Let $(U_n)$ be a countable $\pi$-basis for $K$ and $(x_n)$ be a sequence of UBCP points for $X$ so that $\|x_n\|=r^{**}$ and $S_X\subset \cup B(x_n,r^{**}-r^{*})$ . Let $(f_n)$ be a sequence of functions in $C_0(K)$ so that $\|f_n\|=1, f_n\geq 0$ and $\supp(f_n)\subset U_n$. We will show that $(f_n x_k)_{n, k}$ is a set of UBCP points for $C_0(K, X)$. Let $g$ be any function in $C_0(K, X)$ with $\|g\|=1$. Since $g$ is continuous and $K$ is compact, there is an $x_g=g(\tau_0)$ for some $\tau_0\in K$ so that $\|x_g\|=1$. There is an $t\in\mathbb{N}$ such that $\|x_t-x_g\|\leq r^{**}- r^{*}$. Consider the open set $g^{-1}(B^\circ(x_g, \frac{r^{*}}{2}))$. Since $(U_n)$ is a $\pi$-basis of $K$, there exists $m\in\mathbb{N}$ such that \[U_m\subset g^{-1}(B^\circ(x_g, \frac{r^{*}}{2})).\] So we have
\begin{equation*}
	\begin{split}
		\|f_mx_t-g\|%&=\max\{\sup_{\tau\in U_m}\|f_m(\tau)x_t-g(\tau)\|, \sup_{\tau\in U_m^c} \|f_m(\tau)x_t-g(\tau)\|\}\\
		&=\max\{\sup_{\tau\in U_m}\|f_m(\tau)x_t-g(\tau)\|, \sup_{\tau\in U_m^c} \|g(\tau)\|\}\\
		&\leq\max\{\sup_{\tau\in U_m}(\|f_m(\tau)x_t-x_g\|+\|x_g-g(\tau)\|), 1\}\\
		&\leq\max\{\sup_{\tau\in U_m}(\|f_m(\tau)x_t-x_g\|+\frac{r^{*}}{2}), 1\}
	\end{split}
\end{equation*}
If $f_m(\tau)=1$, then \[\|f_m(\tau)x_t-x_g\|=\|x_t-x_g\|\leq r^{**}-r^{*}.\] If $0\leq f_m(\tau)<1$, then
\begin{equation*}
	\begin{split}
		\|f_m(\tau)x_t-x_g\|&=\|f_m(\tau)(x_t-x_g)+(f_m(\tau)-1)x_g\|\\
		&\leq\|f_m(\tau)(x_t-x_g)\|+\|(f_m(\tau)-1)x_g\|\\
		&=f_m(\tau)\|x_t-x_g\|+(1-f_m(\tau))\\
		&\leq f_m(\tau) (r^{**}-r^{*})+(1-f_m(\tau))r^{**}-r^{*}\\
		&<r^{**}-r^{*}
	\end{split}
\end{equation*}
Therefore we get
\[
\|f_mx_t-g\|\leq \max\{(r^{**}-r^{*}+\frac{r^{*}}{2}), 1\}=r^{**}-\frac{r^{*}}{2}.
\]

\end{proof}

\section{BCP on Non-Commutative Spaces of Operators $\cB(X,Y)$}

Using the definition of the BCP, we have the following lemma, the proof of which will be omitted.

\begin{lemma}\label{lem3.1}
	Let $X$, $Y$ and $Z$ be normed spaces and assume that $X\subset Y\subset Z$. If the (uniform) BCP points of $Z$ can be chosen in $X$, then both $X$ and $Y$ have the (uniform) BCP.
\end{lemma}

Now we show that $\mathcal B(\ell_p)$ $1<p<\infty$ has the BCP. Actually, every subspace containing finite rank operators in $\mathcal B(X, \ell_p)$ has the BCP.

\begin{thm}\label{thm3.2}
Let $X$ be a Banach space with $X^*$ separable. Then for $1<p<\infty$, any subspace $E$ of $B(X, \ell_p)$ containing $\cF(X, \ell_p)$ has the UBCP.
\end{thm}

\begin{proof}
By Lemma \ref{lem3.1}, it is enough to show that $\cB(X, \ell_p)$ has the UBCP. Let $(e_n)$ be the standard unit vector basis for $\ell_p$ and $(e_n^*)$ be the standard unit vector basis for $\ell_q$, where $1/p+1/q=1$. Since $X^*$ is separable, there exists a countable dense subset $(x_n^*)$ in the unit closed ball of $X^*$. Let $\lambda>1$. We will show that the countable set $\mathcal{A}$ defined by
\[ \Big\{ \lambda \sum_{i=1}^{k} x_{m_i}^*\otimes e_i \ :\frac{4-(1-\frac{1}{(2\lambda)^p})^{\frac{1}{p}}}{3} >\| \sum_{i=1}^{k} x_{m_i}^*\otimes e_i  \|>  \frac{2+(1-\frac{1}{(2\lambda)^p})^{\frac{1}{p}}}{3},  k, {m_i}\in\mathbb{N} \Big\}\]
is a set of UBCP points for $\mathcal B(X, \ell_p)$. Let $T$ be an operator in $\mathcal B(X, \ell_p)$ with $\|T\|=1$. Then $T=\sum_n e_n^* T\otimes e_n,$ where the series converges in the strong operator topology. Since $\|T\|=1$, there is a $t_0\in\mathbb{N}$ so that
$$
\max\{\frac{1}{2}, \frac{1}{\lambda^{p-1}}\}<\theta=\|\sum_1^{t_0}e_i^* T\otimes e_i\|\leq 1.
$$
Choose $\epsilon>0$ such that
$$
\epsilon<\frac{1-\big(1-\frac{1}{(2\lambda)^p}\big)^{\frac{1}{p}}}{2}.
$$
Since $\big\| \frac{e_i^*T}{\theta} \big\| \leq 1$ for each $i\in\mathbb{N}$, we can find  $m_i\in\mathbb{N}$ so that
$$
 \big\|x_{m_i}^*-\frac{e_i^*T}{\theta}\big\|<\min\Big(\frac{\epsilon}{t_0},\frac{1-(1-\frac{1}{(2\lambda)^p})^{\frac{1}{p}}}{3{t_0}}\Big).
$$
By triangular inequality and a regular computation, it is not difficult to see that
$\lambda\sum_1^{t_0} x_{m_i}^*\otimes e_i$ is in set $\mathcal{A}$. We will prove that \[\|T-\lambda\sum_1^{t_0}x_{m_i}^* T\otimes e_i\|<\lambda\big((1-\frac{\theta^p}{\lambda^p})^{\frac{1}{p}}+\epsilon\big).\] Actually we have
\begin{equation*}
\begin{split}
&\|T-\lambda \sum_1^{t_0}x_{m_i}^* T\otimes e_i\|\\
&=\|\sum_n e_n^* T\otimes e_n-\lambda \sum_1^{t_0}x_{m_i}^* T\otimes e_i\ \|\\
&=\|(1-\frac{\lambda}{\theta}) \sum_1^{t_0}x_{m_i}^* T\otimes e_i\ +\sum_{i> {t_0}} e_i^* T\otimes e_i \| + \|  \lambda\sum_1^{t_0} (\frac{e_i^*T}{\theta}\otimes e_i-x_{m_i}^*\otimes e_i)\|\\
&\leq \|(1-\frac{\lambda}{\theta}) \sum_1^{t_0} e_{m_i}^*T\otimes e_i+\sum_{i>{t_0}} e_i^* T\otimes e_i\|+ +\sum_1^{t_0}\| \lambda(\frac{e_i^*T}{\theta}-x_{m_i}^*)\otimes e_i\|\\
&<\sup_{x\in B_X}\|(1-\frac{\lambda}{\theta}) \sum_1^{t_0} e_{m_i}^*T\otimes e_i (x) +\sum_{i>{t_0}} e_i^* T(x) e_i\|+\lambda\epsilon\\
&= \sup_{x\in B_X}\big(|1-\frac{\lambda}{\theta}|^p |\sum_1^{t_0} e_{m_i}^*T\otimes e_i (x)|^p+\sum_{i>{t_0}} |e_i^*T(x)|^p\big)^{\frac{1}{p}}+\lambda\epsilon\\
&\leq \sup_{x\in B_X}\big((|1-\frac{\lambda}{\theta}|^p-1) \sum_1^{t_0} |e_i^*T(x)|^p+\sum_{i} |e_i^*T(x)|^p\big)^{\frac{1}{p}}+\lambda\epsilon\\
&\leq \sup_{x\in B_X}\big((|1-\frac{\lambda}{\theta}|^p-1) \sum_1^{t_0} |e_i^*T(x)|^p+\|Tx\|^p\big)^{\frac{1}{p}}+\lambda\epsilon\\
&\leq \big((|1-\frac{\lambda}{\theta}|^p-1) \theta^p+1\big)^{\frac{1}{p}}+\lambda\epsilon\\
&= \big((|\theta-\lambda|^p-\theta^p+1\big)^{\frac{1}{p}}+\lambda\epsilon\\
&= \big(\lambda^p( (1-\frac{\theta}{\lambda})^p-\frac{\theta^p}{\lambda^p}+\frac{1}{\lambda^p})\big)^{\frac{1}{p}}+\lambda\epsilon\\
&\leq\big(\lambda^p( (1-\frac{\theta}{\lambda})-\frac{\theta^p}{\lambda^p}+\frac{1}{\lambda^p})\big)^{\frac{1}{p}}+\lambda\epsilon\\
&=\big(\lambda^p( (1-\frac{\theta^p}{\lambda^p})+(\frac{1}{\lambda^p}-\frac{\theta}{\lambda}))\big)^{\frac{1}{p}}+\lambda\epsilon\\
&<\big(\lambda^p( (1-\frac{\theta^p}{\lambda^p})\big)^{\frac{1}{p}}+\lambda\epsilon\\
&=\lambda(1-\frac{\theta^p}{\lambda^p})^{\frac{1}{p}}+\lambda\epsilon\\
&=\lambda\big((1-\frac{\theta^p}{\lambda^p})^{\frac{1}{p}}+\epsilon\big)\\
&<\lambda((1-\frac{1}{(2\lambda)^p})^{\frac{1}{p}}+\frac{1-(1-\frac{1}{(2\lambda)^p})^{\frac{1}{p}}}{2})\\
&=\lambda\cdot \frac{1+(1-\frac{1}{(2\lambda)^p})^{\frac{1}{p}}}{2}.
\end{split}
\end{equation*}
This finishes the proof.

\end{proof}

\begin{cor} Let $1<p<\infty$.
Then every subspace of $\mathcal B(\ell_p)$ containing $\mathcal F(\ell_p)$ has the UBCP.
\end{cor}

For a complex separable Hilbert space $H$, we have a brief proof to obtain that any complex operator subspace $E$ of $\mathcal B(H)$
containing $\cF(H)$ has the UBCP,
which is the non-commutative analogue of $\ell_\infty$ with the UBCP. It replies on
complex operator algebra fundamental techniques. Indeed,
Let $\{b_n\}_n$ be a dense sequence in the unit sphere $S_1(H)$.
Then $\{\lambda b_n\otimes b_m\}_{n,m}$ is the UBCP points
with $1<\lambda <2$, and \[S_1(E)\subset \bigcup_{n,m} \mathrm{B}(\lambda b_n\otimes b_m, 1)\]
is the corresponding uniform countable ball covering.
Let $A\in S_1(E)$ with the Polar Decomposition $A=U|A|$ where $U$ is a unitary and $|A|$ is positive.
By the Spectral Theorem, there is $f\ge 0$ in some $L^\infty(\mu)$-space and a unitary $V: L^2(\mu)\rightarrow \ell_2$
such that $|A|=V M_f V^*.$
Then for any $\epsilon>0$ there is finite orthogonal projections
$P_k$ and $0<p_k\le 1$
such that \[\big\||
A|-\sum^N_{k=1} p_k P_k\big\|<\epsilon.\]
Without loss of generality, we assume that $|p_1-1|<\epsilon$ and $P_1$
is rank one, that is $P_1=x\otimes x$ with $x\in S_1(H)$.
Thus, 
\begin{eqnarray*}
U\Big(\sum^N_{k=1} p_k P_k\Big)&=&U(p_1 x\otimes x+\sum^N_{k=2} p_k P_k)\\
&=&p_1 U(x)\otimes x +
\sum^N_{k=2} p_k U P_k.
\end{eqnarray*}
There are $b_n,b_m$ such that $\|b_n-U(x)\|<\epsilon$
and $\|b_m-x\|<\epsilon$, then we have
\begin{eqnarray*}
  \|A-\lambda b_n\otimes b_m\| &\le&
  \Big\|U(\sum^N_{k=1} p_k P_k)-\lambda b_n\otimes b_m\Big\| +\epsilon \\
 %  &=& \Big\|p_1 x\otimes x+\sum^N_{k=2} p_k P_k-\lambda U^{-1}(b_n)\otimes b_m\Big\| +\epsilon \\
   &=& \Big\|(p_1-\lambda)x\otimes x+\sum^N_{k=2} p_k P_k+\lambda(x\otimes x- U^{-1}(b_n)\otimes b_m) \Big\| + \epsilon\\
   &\le& \Big\|(p_1-\lambda)x\otimes x+\sum^N_{k=2} p_k P_k\Big\| + (2\lambda+1)\epsilon \\
   &\le& 1+   (2\lambda+2)\epsilon %\\
%   &<& \lambda %= \|\lambda b_n\otimes b_m\|
\end{eqnarray*}
\begin{remark}
It is well known that $L_\infty[0, 1]$ is completely isometrically isomorphic to a completely 1-complemented subspace of $\mathcal B(H)$
in the operator space sense (\cite{GMS}).
Since $L_\infty[0, 1]$ fails the BCP \cite{LZh1}, Theorem \ref{thm3.2} tells us that $\cB(H)$ and $L_\infty[0, 1]$ form a counterexample for Question 2 in the introduction.
\end{remark}

Next, we present some necessary conditions for $\cB(X,Y)$ with the BCP.

\begin{thm}\label{thm3.1}
	Let $X$ and $Y$ be Banach spaces. If $\mathcal B(X,Y)$ has the BCP, then both $X^*$ and $Y$ have the BCP.
\end{thm}

\begin{proof}

Suppose that $S_{\mathcal B(X,Y)}\subset\bigcup\limits_{n=1}^\infty B(T_n,r_n)$, where $r_n<\|T_n\|$ for all $n\in\mathbb N$.
For every $T_n$, there is $x_n\in S_{X}$ such that $\|T_n{x_n}\|>r_n$. It follows from the completeness of $Y$ and the Baire Category Theorem that there exists \[g\in S_{X^*}\backslash\bigcup\limits_{n=1}^\infty x_n^\perp, \mbox{ where } x_n^\perp=\{f\in X^*: f(x_n)=0\}.\] We show that $\left\{\frac{T_nx_n}{g(x_n)}:n\in\mathbb N\right\}$ is a collection of BCP points of $Y$.
	For $y\in S_Y$, consider the operator $R_y: X\rightarrow Y$ defined by for $x\in X$, \[R_y(x)=g(x)y.\] There is $n_0\in\mathbb N$ such that $\|R_y-T_{n_0}\|\leq r_{n_0}$. Therefore,
%	\[\|T_yf-T_{n_0}^*f\|=\|y(f)g-T_{n_0}^*f\|\leq r_{n_0}\|f\|.\]
%	Since $x_{n_0}\in S_{X}$,
%	\[\|y(f)x_{n_0}(g)-(T_{n_0}x_{n_0})(f)\|\leq\|y(f)g-T_{n_0}^*f\|\leq r_{n_0}\|f\|<\|T_{n_0} x_{n_0}\|\cdot\|f\|,\]
%	i.e.
	\[\Big\|y-\frac{T_{n_0}x_{n_0}}{g(x_{n_0})}\Big\|= \frac{\|R_y x_{n_0}-T_{n_0}x_{n_0}\|}{|g(x_{n_0})|}
\leq\frac{r_{n_0}}{|g(x_{n_0})|}
<\Big\|\frac{T_{n_0}x_{n_0}}{g(x_{n_0})}\Big\|.\]

	%Suppose that $S_{\mathcal B(X,Y)}\subset\bigcup\limits_{n=1}^\infty B(T_n,r_n)$, where $r_n<\|
%T_n\|
%$ for all $n\in\mathbb N$.
Let $T_n^*:Y^*\rightarrow X^*$ denote the adjoint operator of $T_n$. For every $T_n^*$, there is $g_n\in S_{Y^*}$ such that
$\|T_n^*g_n\|>r_n$. Similarly, by the Baire Category Theorem, there exists \[y\in S_Y\backslash\bigcup\limits_{n=1}^\infty\ker g_n.\] We show that $\left\{\frac{T_n^*g_n}{g_n(y)}:n\in\mathbb N\right\}$ is a collection of BCP points of $X^*$.
	For $f\in S_{X^*}$, consider the operator $L_f:X\rightarrow Y$ defined by for $x\in X$,\[L_f(x)=f(x)y.\] There is $n_0\in\mathbb N$ such that
$\|L_f-T_{n_0}\|\leq r_{n_0}$. Therefore,
\begin{eqnarray*}	
\|f(x)g_{n_0}(y)-(T_{n_0}^*g_{n_0})(x)\|
&\leq&\|L_f(x)-T_{n_0}x\|\\
&\leq& r_{n_0}\|x\|<\|T_{n_0}^*g_{n_0}\|\cdot\|x\|,
\end{eqnarray*}
	that is,
	\[\Bigg\|f-\frac{T_{n_0}^*g_{n_0}}{g_{n_0}(y)}\Bigg\|\leq \frac{r_{n_0}}{|g_{n_0}(y)|}<\Bigg\|\frac{T_{n_0}^*g_{n_0}}{g_{n_0}(y)}\Bigg\|.\]
Then we complete the proof.
%Similarly, we  prove that $Y$ has the BCP.
	
\end{proof}

Although $\ell_{\infty}$ has the BCP, it was proved in \cite{LZh1} that $L_{\infty}[0, 1]$ fails the BCP. So the following corollary is an immediate consequence of Theorem \ref{thm3.1}.

\begin{cor}
$\mathcal B(L_1[0,1])$ fails the BCP.
\end{cor}

For the BCP of $\cB(\ell_1)$ and $\cB(c_0)$ we need the stability of vector-valued sequence spaces and tensor products of Banach spaces.
Let $X$ and $Y$ be two Banach spaces. We use $X\hat{\otimes}_\pi Y$ to denote the projective tensor product. If $X$ is isometrically isomorphic to $Y$, then we write $X\cong Y$.
Let $X \otimes Y$ be the tensor product of the Banach space $X$ and $Y$. The projective norm $\|\cdot\|_\pi$ on $X \otimes Y$ is defined by:
$$
\|u\|_{\pi}=\inf \left\{\sum_{i=1}^{n}\left\|x_{i}\right\|\left\|y_{i}\right\|: u=\sum_{i=1}^{n} x_{i} \otimes y_{i}\right\}
$$
We will use $X \otimes_{\pi} Y$ to denote the tensor product $X \otimes Y$ endowed with the projective norm $\|\cdot\|_{\pi} .$ Its completion will be denoted by $X \hat{\otimes} Y$. For any Banach spaces $X$ and $Y$, we have the identification:
$$
(X \hat{\otimes} Y)^{*}\cong B\left(X, Y^{*}\right)
$$
\begin{thm}[\cite{LZh1}, Theorem 2.8 and Remark 2.2]\label{thm1.1}
	\it Let $\{X_k\}$ be a sequence of normed speces. Then $\left(\sum\oplus X_k\right)_{\ell_\infty}$ has the BCP if and only if each $X_k$ has the BCP, and the BCP points of $\left(\sum\oplus X_k\right)_{\ell_\infty}$ can be chosen in $\left(\sum\oplus X_k\right)_{c_0}$.
\end{thm}

\begin{thm}
	Let $X$ be a Banach space.
	\begin{itemize}
	\item[(i)] $\mathcal B(X, \ell_\infty)$ has the BCP if and only if $X^*$ has the BCP.
	\item[(ii)] $\mathcal B(X, c_0)$ has the BCP if and only if $X^*$ has the BCP.
	\end{itemize}
\end{thm}

\begin{proof}
	(i) Since
\[\mathcal B(X,\ell_\infty)\cong (X\hat{\otimes}_\pi \ell_1)^*\cong
(\ell_1(X))^*\cong\ell_\infty(X^*),\]
the statement follows immediately from Theorem \ref{thm1.1} and Theorem \ref{thm3.1}.	

	(ii) Since it is easy to prove that $\left(\sum\oplus X^*\right)_{c_0}\subset\mathcal B(X,c_0)\subset\left(\sum\oplus X^*\right)_{\ell_\infty}$, it follows from Theorem 1.1 and Lemma 3.1 that $\mathcal B(X,c_0)$ has the BCP.
\end{proof}

Since $\mathcal B(X,\ell_\infty)\cong (X\hat{\otimes}_\pi \ell_1)^*\cong B(\ell_1, X^*)$, we have the following corollaries.
\begin{cor}
Let $X$ be a normed space.
	Then $\mathcal B(\ell_1, X^*)$ has the BCP if and only if $X^*$ has the BCP.
\end{cor}

\begin{cor}
$\mathcal B(\ell_1)$ and $\mathcal B(c_0)$ both have the BCP.
\end{cor}

\section*{Acknowledgments} The authors would like to express appreciation to Longyun Ding and Chi-Keung Ng for helpful discussions.

\end{document}